\documentclass[a4paper,11pt,twoside]{amsart}
\usepackage[english]{babel}
\usepackage[utf8]{inputenc}

\usepackage[a4paper,inner=3cm,outer=3cm,top=4cm,bottom=4cm,pdftex]{geometry}
\usepackage{fancyhdr}
\usepackage{color}
\usepackage{bold-extra}
\usepackage{ mathrsfs }

\usepackage{comment}
\usepackage{graphics}
\usepackage{aliascnt}
\usepackage[pdftex,citecolor=green,linkcolor=red]{hyperref}

\usepackage{amsmath}
\usepackage{amsfonts}
\usepackage{amssymb}
\usepackage{amsthm}
\usepackage{comment}
\usepackage{mathtools}
\usepackage{enumerate}
\usepackage{enumitem} 

\newtheorem{theorem}{Theorem}[section]

\newtheorem{lemma}[theorem]{Lemma}
\newtheorem{proposition}[theorem]{Proposition}
\theoremstyle{definition}

\newtheorem{conjecture}[theorem]{Conjecture}
\newtheorem{question}[theorem]{Question}
\numberwithin{equation}{section}

\newcommand\eps{\varepsilon}

\newcommand\R{\mathbb{R}}
\newcommand\Z{\mathbb{Z}}

\newcommand\C{\mathbb{C}}

\newcommand{\wh}{\widehat}

\newcommand{\dd}{\mathrm{d}}

\begin{document}
\title{Additive energies of subsets of discrete cubes}

\author[Shao]{Xuancheng Shao}
\address{Department of Mathematics, University of Kentucky\\
715 Patterson Office Tower\\
Lexington, KY 40506\\
USA}
\email{xuancheng.shao@uky.edu}
\thanks{XS was supported by NSF grant DMS-2200565. Thanks to Ali Alsetri for pointing out the reference \cite{Mazur} and to Andrew Granville for helpful discussions.}


\maketitle

\begin{abstract}
For a positive integer $n \geq 2$, define $t_n$ to be the smallest number such that the additive energy $E(A)$ of any subset $A \subset \{0,1,\cdots,n-1\}^d$ and any $d$ is at most $|A|^{t_n}$. Trivially we have $t_n \leq 3$ and 
$$ t_n \geq 3 - \log_n\frac{3n^3}{2n^3+n} $$
by considering $A = \{0,1,\cdots,n-1\}^d$. In this note, we investigate the behavior of $t_n$ for large $n$ and obtain the following non-trivial bounds:
$$ 3 - (1+o_{n\rightarrow\infty}(1)) \log_n \frac{3\sqrt{3}}{4} \leq t_n \leq 3 - \log_n(1+c), $$
where $c>0$ is an absolute constant.
\end{abstract}

\section{Introduction}

Let $A \subset G$ be a finite subset of an abelian group $G$. The additive energy $E(A)$ of $A$ is defined to be the number of additive quadruples in $A$:
$$ E(A) = \#\{(a_1,a_2,a_3,a_4) \in A^4: a_1+a_2=a_3+a_4\}. $$
Trivially we have $|A|^2 \leq E(A) \leq |A|^3$. A central theme in additive combinatorics is to understand the structure of those sets $A$ whose additive energy $E(A)$ is close to its trivial upper bound $|A|^3$. The famous Balog-Szemeredi-Gowers theorem and Freiman's theorem are both results in this direction. See \cite{TaoVu} for precise statements of these results and their proofs.

In this paper we study upper bounds for $E(A)$ when $A$ lies in certain subsets of $\Z^d$ for potentially large $d$. 
For a positive integer $n \geq 2$, define $t_n$ to be the smallest number such that $E(A) \leq |A|^{t_n}$ for all subsets $A \subset \{0,1,\cdots,n-1\}^d$ and all positive integers $d$.
One can calculate that
\begin{equation*}
\begin{split} 
E(\{0,1,\cdots,n-1\}) &= \sum_{s \in \Z} |\{(a,b): s=a+b, 0 \leq a,b\leq n-1\}|^2 \\
&= 1^2 + 2^2 + \cdots + n^2 + (n-1)^2 + \cdots + 1^2 = \frac{2n^3+n}{3}
\end{split}
\end{equation*}
and that
$$
E(\{0,1,\cdots,n-1\}^d) = E(\{0,1,\cdots,n-1\})^d = \left(\frac{2n^3+n}{3}\right)^d. 
$$
Thus we have the trivial bounds
\begin{equation}\label{trivial-bounds}
3 \geq t_n \geq \log_n \frac{2n^3+n}{3} = 3 - \log_n \frac{3n^3}{2n^3+n}.
\end{equation}
 It is known \cite[Theorem 7]{KaneTao} that $t_2 = \log_26$ so that the lower bound in \eqref{trivial-bounds} for $t_2$ is sharp. For $n = 3$, it was proved in \cite{DGIM} that
$$ t_3 \geq 2\log_2 2.5664 \geq 2.71949. $$
See \cite[Proposition 6]{DGIM} and its proof in \cite[Section 4.3]{DGIM}. In particular, this implies that the trivial lower bound $t_3 \geq \log_319 \approx 2.68$ in \eqref{trivial-bounds} is not sharp. Our main goal is to explore the behavior of $t_n$ for large $n$. 

\begin{theorem}\label{main-thm}
Let $n \geq 2$ be a positive integer. Then for some absolute constant $c>0$ we have
$$ 3 - (1+o_{n\rightarrow\infty}(1)) \log_n \frac{3\sqrt{3}}{4} \leq t_n \leq 3 - \log_n (1+c), $$
where $o_{n\rightarrow\infty}(1)$ denotes a quantity that tends to $0$ as $n\rightarrow\infty$.
\end{theorem}

Unfortunately the lower bound in Theorem \ref{main-thm} is only meaningful for $n$ sufficiently large. To complement that, we also prove the following result which is valid for every $n \geq 3$.

\begin{theorem}\label{main-thm2}
For any positive integer $n \geq 3$, we have
$$ t_n > \log_n E(\{0,1,\cdots,n-1\}) = \log_n \frac{2n^3+n}{3}. $$
\end{theorem}

A key tool for the proof of both theorems comes from \cite{DGIM}, which allows us to pass from studying subsets in $\Z^d$ to studying functions on $\Z$. In Section \ref{sec:outline} we will describe this tool, outline the proofs, and make some remarks on further directions. The lower bound and the upper bound in Theorem \ref{main-thm} will be proved in Sections \ref{sec:lower-bound} and  \ref{sec:upper-bound}, respectively. Theorem \ref{main-thm2} will be proved in Section \ref{sec:thm2}.

\section{Proof outline}\label{sec:outline}

For a finitely supported function $f: \Z \rightarrow \C$, we define its Fourier transform $\wh{f}: \R/\Z \rightarrow \C$ by the formula
$$ \wh{f}(\theta) = \sum_{a \in \Z} f(a) e(-a\theta), $$
where $e(x) = e^{2\pi ix}$. For $p,q \geq 1$, the $L^p$-norm of $\wh{f}$ and the $\ell^q$-norm of $f$ are defined by
$$ \|\wh{f}\|_p = \left(\int_0^1 |\wh{f}(\theta)|^p \dd\theta\right)^{1/p}, \ \ \|f\|_q = \left(\sum_{a \in \Z} |f(a)|^q\right)^{1/q}. $$
For two finitely supported functions  $f, g: \Z \rightarrow \C$, their convolution $f*g: \Z\rightarrow \C$ is defined by
$$ f*g(s) = \sum_{a \in \Z} f(a) g(s-a). $$
We have the identities
$$ \|\wh{f}\|_4^4 = \|f*f\|_2^2 = \sum_{a, b, c \in \Z} f(a) f(b) \overline{f(c) f(a+b-c)}. $$
Thus if $f = 1_A$ is the indicator function of a finite subset $A \subset \Z$, then 
$$ E(A) = \|1_A*1_A\|_2^2 = \|\wh{1_A}\|_4^4. $$
In Section \ref{sec:lower-bound} we will also need to utilize Fourier transforms of functions on $\R$. For a piecewise continuous function $g: \R\rightarrow \C$ which  has bounded support, we define its Fourier transform $\wh{g}: \R\rightarrow \C$ by the formula
$$ \wh{g}(y) = \int_{-\infty}^{+\infty} f(x) e(-xy) \dd x. $$
For two such functions $g,h$, we define their convolution $g*h: \R\rightarrow \C$ by 
$$ g*h(z) = \int_{-\infty}^{+\infty} g(x) h(z-x) \dd x. $$
We have the identities
$$ \|\wh{g}\|_4^4 = \|g*g\|_2^2 = \int\int\int g(x_1)g(x_2) \overline{g(x_3)g(x_1+x_2-x_3)} \dd x_1 \dd x_2 \dd x_3. $$

The machinery developed in \cite[Section 4]{DGIM} plays a key role in our proof. We summarize their result in the following proposition. Recall the definition of $t_n$ from the introduction.

\begin{proposition}\label{prop:tool}
Let $n \geq 2$ be a positive integer. We have $t_n = 4/q_n$, where $q_n$ is the largest value of $q$ such that the inequality $\|\wh{f}\|_4 \leq \|f\|_q$ holds for any function $f: \Z\rightarrow \R$ which is supported on an interval of length $n$.
\end{proposition}

\begin{proof}
This is essentially \cite[Proposition 21]{DGIM}. First observe that, by translation, we may restrict to those functions $f: \Z\rightarrow\R$ supported on $A = \{0,1,\cdots,n-1\}$ in the definition of $q_n$. Then, in the language of \cite[Definition 14]{DGIM}, $q_n$ is the largest value of $q$ such that 
\begin{equation}\label{eq-DE} 
\operatorname{DE}_{\ell^q\rightarrow L^4}(A) \leq 1,
\end{equation}
where $\operatorname{DE}_{\ell^q\rightarrow L^4}(A)$ is the operator norm of the linear map $\ell^q(A) \rightarrow L^4(\R/\Z)$ defined by the Fourier transform $f \mapsto \wh{f}$.
By \cite[Proposition 21]{DGIM}, \eqref{eq-DE} is equivalent to the statement that an inequality of the form 
$$ E(X) \leq |X|^{4/q} $$
holds for all subsets $X \subset A^d$ and $d \geq 1$. It follows that $t_n = 4/q_n$ by the definition of $t_n$.
\end{proof}

We remark that, by the Hausdorff-Young inequality, we always have
$$ \|\wh{f}\|_4 \leq \|f\|_{4/3}. $$
Hence $q_n \geq 4/3$ and this recovers the trivial bound $t_n \leq 3$.
Moreover, the $\ell^{4/3}$-norm and the $\ell^q$-norm for $q > 4/3$ are related by the inequalities
\begin{equation}\label{eq:q-4/3} 
\|f\|_q \leq \|f\|_{4/3} \leq |\operatorname{supp} f|^{3/4-1/q} \cdot \|f\|_q, 
\end{equation}
where $|\operatorname{supp}f|$ denotes the size of the support of $f$. 

In view of Proposition \ref{prop:tool}, the lower and upper bounds in Theorem \ref{main-thm} follow from Propositions \ref{prop:lower-bound} and \ref{prop:upper-bound} below, respectively. In the remainder of this section, we discuss the main ideas behind the proofs of these two propositions and make some remarks about the quality of our bounds.

\subsection{Lower bound for $t_n$}

In view of Proposition \ref{prop:tool}, the lower bound for $t_n$ in Theorem \ref{main-thm} is equivalent to the following proposition.

\begin{proposition}\label{prop:lower-bound}
Let $\eps > 0$ and let $n$ be sufficiently large in terms of $\eps$. Let
$$ q = \frac{4}{3 - (1+\eps)\log_n\frac{3\sqrt{3}}{4}}. $$
There exists a function $f: \Z\rightarrow \R$ which is supported on an interval of length $n$, such that $\|\wh{f}\|_4 > \|f\|_q$. 
\end{proposition}

Our motivation for the construction of $f$ in Proposition \ref{prop:lower-bound} comes from the Babenko-Beckner inequality \cite{Babenko, Beckner}, a sharpened form of the Hausdorff-Young inequality for functions on $\R$ (and more generally on $\R^d$). It asserts that for any function $g: \R\rightarrow \R$ we have
\begin{equation}\label{Beckner} 
\|\wh{g}\|_4 \leq \left(\frac{4\sqrt{3}}{9}\right)^{1/4} \|g\|_{4/3}. 
\end{equation}
Moreover, equality is achieved when $g$ is the Gaussian function $g(x) = e^{-x^2}$. In other words, Gaussian functions (and similarly their dilated versions) maximize the $\wh{L}_4$-norm, if we hold the $\ell^{4/3}$-norm fixed. If we take $g(x) = e^{-x^2/A}$ with $A \approx n^2$ (so that $g$ is essentially supported on an interval of length $\approx n$), then direct computations show that
$$ \frac{\|g\|_{4/3}}{\|g\|_q} = c A^{\frac{1}{2}(\frac{3}{4}-\frac{1}{q})}, $$
where $c$ is an explicit constant depending on $q$ and $c\approx 1$ when $q \approx 4/3$. By our choice of $A$ and $q$, we have
$$ A^{\frac{1}{2}(\frac{3}{4}-\frac{1}{q})} \approx n^{\frac{3}{4}-\frac{1}{q}} = n^{\frac{1}{4}(1+\eps)\log_n\frac{3\sqrt{3}}{4}} \approx \left(\frac{3\sqrt{3}}{4}+c\right)^{1/4} $$
for some constant $c=c(\eps) > 0$.
Hence this function $g(x)$ satisfies 
$$ \|\wh{g}\|_4 = \left(\frac{4\sqrt{3}}{9}\right)^{1/4} \|g\|_{4/3} \approx \left(\frac{4\sqrt{3}}{9}\right)^{1/4}  \left(\frac{3\sqrt{3}}{4} + c\right)^{1/4}\|g\|_q > \|g\|_q. $$
If we define $f: \Z\rightarrow \R$ by sampling the values of $g(x)$ at integral points, then we may expect that
$$ \|\wh{f}\|_4 \approx \|\wh{g}\|_4, \ \ \|f\|_q \approx \|g\|_q, $$
and thus we should also have $\|\wh{f}\|_4 > \|f\|_q$. The details are worked out in Section \ref{sec:lower-bound}.

\subsection{Upper bound for $t_n$}

In view of Proposition \ref{prop:tool}, the upper bound for $t_n$ in Theorem \ref{main-thm} is equivalent to the following proposition.

\begin{proposition}\label{prop:upper-bound}
Let $n \geq 2$ be a positive integer and let $f: \Z\rightarrow \R$ be a function which is supported on a set of size $n$. Let
$$  q = \frac{4}{3 - \log_n(1+c)} $$
for some sufficiently small absolute constant $c > 0$. Then $\|\wh{f}\|_4 \leq \|f\|_q$.
\end{proposition}

The starting point of our proof of Proposition \ref{prop:upper-bound} is the inequality
\begin{equation}\label{eq:HY}
\|\widehat{f}\|_4 \leq \|f\|_{4/3} 
\end{equation}
which follows from the Hausdorff-Young inequality or Young's convolution inequality. By H\"{o}lder's inequality  (see \eqref{eq:q-4/3}) and the definition of $q$, we have
$$ \|f\|_{4/3} \leq n^{3/4 - 1/q} \|f\|_q = (1+c)^{1/4} \|f\|_q. $$
Thus the proof is already complete unless
$$ \|\widehat{f}\|_4 \geq (1+c)^{-1/4} \|f\|_{4/3}, $$
and thus a key part of our argument is to analyze when equality almost holds in \eqref{eq:HY}. Note that equality holds exactly in \eqref{eq:HY} when $f$ is supported on a singleton set. We prove in Proposition \ref{prop:HY} that if equality almost holds in \eqref{eq:HY}, then $f$ is well approximated by a function $f_0$ which is supported on a singleton set, up to an error $g$ which is small in $\ell^{4/3}$-norm. Clearly the function $f_0$ satisfies $\|\wh{f_0}\|_4 = \|f_0\|_q$. The remaining task is then to show that the error $g$ can only swing the inequality in the desired direction. The details are carried out in Section \ref{sec:upper-bound}.

We remark that Proposition \ref{prop:HY} is not new. In fact, it is a special case of \cite[Theorem 1.2]{CC} (see also \cite{Christ} for an analogous result in Euclidean spaces) and of \cite[Proposition 5.4]{EisnerTao}. As it turns out, our proof idea is the same as that in \cite{EisnerTao}, which in turn has its origin from \cite{Fournier}. For completeness, we still give a self-contained  proof of it in Section \ref{sec:upper-bound}.

\subsection{Questions and speculations}

Our proof of the lower bounds for $t_n$ is not constructive, which motivates the question of constructing explicit subsets of $\{0,1,\cdots,n-1\}^d$ with large additive energies.

\begin{question}
For sufficiently large $n$, construct a subset $A \subset \{0,1,\cdots,n-1\}^d$ for some $d$ such that $E(A) \geq |A|^t$, where
$$ t = 3 - (1+o_{n\rightarrow\infty}(1)) \log_n \frac{3\sqrt{3}}{4}. $$
\end{question}

A possible candidate for such a set $A$ is the set of lattice points in a $d$-dimensional ball $B_d \subset \R^d$ (with an appropriate choice of $d$ and an appropriate center and radius). This choice is motivated by results in \cite{Shao2} which implies, roughly speaking, that such a set $A$ maximizes the additive energy among all genuinely $d$-dimensional subsets of $\Z^d$ of a given cardinality. Moreover, $E(A) \approx E(B_d)$  and it follows from the computations in \cite[Section 3.1]{Mazur} that
$$ E(B_d) = \left(\frac{4\sqrt{3}}{9} + o_{d\rightarrow\infty}(1)\right)^d |B_d|^3, $$
where $|B_d|$ denotes the Lebesgue measure of $B_d$.

Next we speculate the asymptotic behavior of $t_n$ as $n\rightarrow\infty$. Note that if $g: \R\rightarrow \R$ is a (continuous) function supported on an interval of length $n$ and
$$ q = \frac{4}{3-\log_n\frac{3\sqrt{3}}{4}}, $$
then
$$ \|\wh{g}\|_4 \leq \left(\frac{4\sqrt{3}}{9}\right)^{1/4} \|g\|_{4/3} \leq \left(\frac{4\sqrt{3}}{9}\right)^{1/4}  n^{3/4-1/q} \|g\|_q = \|g\|_q, $$
where the first inequality follows from the Babenko-Beckner inequality \eqref{Beckner} and the second inequality follows from H\"{o}lder's inequality (a continuous version of \eqref{eq:q-4/3}). Based on this, it is perhaps reasonable to conjecture that a similar bound holds for discrete functions.

\begin{conjecture}
Let $\eps > 0$ and let $n$ be sufficiently large in terms of $\eps$. Let
$$  q = \frac{4}{3 - (1-\eps)\log_n\frac{3\sqrt{3}}{4}}. $$
Then for any function $f: \Z\rightarrow \R$ which is supported on an interval of length $n$, we have $\|\wh{f}\|_4 \leq \|f\|_q$.
\end{conjecture}

In particular, the conjecture would imply that 
$$ t_n = 3 - (1+o_{n\rightarrow\infty}(1)) \log_n \frac{3\sqrt{3}}{4}. $$
So perhaps the lower bound in Theorem \ref{main-thm} is sharp up to the error in $o(1)$.

\section{Lower bound for $t_n$}\label{sec:lower-bound}

In this section we prove Proposition \ref{prop:lower-bound}.
Throughout this section, let $\eps>0$ be small and let $n = 2k+1$ be sufficiently large in terms of $\eps$. We will construct a function $f: \Z\rightarrow \R$ supported on $\{-k,\cdots,k\}$ such that $\|\wh{f}\|_4 > \|f\|_q$, where
$$ q = \frac{4}{3-(1+\eps) \log_n \frac{3\sqrt{3}}{4}}. $$

Define $g: \R\rightarrow \R$ by $g(x) = \exp(-x^2/A)$, where $A = k^{2-\eps/10}$. 

\begin{lemma}\label{lem:g-property}
We have $\|\wh{g}\|_4 \geq (1+c\eps)\|g\|_q$ for some absolute constant $c > 0$.
\end{lemma}

\begin{proof}
One can compute that
$$ \wh{g}(y) = (\pi A)^{1/2} e^{-\pi^2 A y^2}, $$
and hence
$$ \|\wh{g}\|_4^4 = (\pi A)^2 \int_{-\infty}^{\infty} e^{-4\pi^2 A y^2} \dd y = \frac{1}{2} (\pi A)^{3/2}. $$
On the other hand, we have
$$ \|g\|_q^q = \int_{-\infty}^{\infty} e^{-qx^2/A} \dd x = \left(\frac{\pi A}{q}\right)^{1/2}. $$
It follows that
$$ \frac{\|\wh{g}\|_4}{\|g\|_q} = \left(\frac{1}{4} q^{4/q} \pi^{3-4/q} A^{3-4/q}\right)^{1/8}. $$
By our choice of $A$, we have
$$ A^{3-4/q} = \exp\left( \left(2-\frac{\eps}{10}\right) (\log k) (1+\eps)\log_n\frac{3\sqrt{3}}{4} \right) \geq \exp\left((2+\eps)\log\frac{3\sqrt{3}}{4}\right) \geq (1+c\eps) \frac{27}{16}. $$
for some absolute constant $c>0$. By choosing $k$ to be sufficiently large in terms of $\eps$, we may ensure that $q$ is sufficiently close to $4/3$ so that
$$ \frac{1}{4} q^{4/q} \pi^{3-4/q} \geq \left(1-\frac{c\eps}{2}\right) \frac{1}{4} \left(\frac{4}{3}\right)^3 = \left(1-\frac{c\eps}{2}\right) \frac{16}{27}.  $$
Combining the two inequalities above, we conclude that
$$ \frac{\|\wh{g}\|_4}{\|g\|_q} \geq \left[(1+c\eps)\left(1-\frac{c\eps}{2}\right)\right]^{1/8} \geq 1 + \frac{c\eps}{100}. $$
\end{proof}

Now we truncate $g$ to have bounded support. Set $M = \lfloor k^{1-\eps/100}\rfloor$. Let $g_M: \R\rightarrow \R$ be the truncation of $g$ defined by
$$ g_M(x) = \begin{cases} g(x) & \text{if } -M \leq x < M, \\ 0 & \text{otherwise.} \end{cases} $$

\begin{lemma}\label{lem:gM-approx-g}
We have $\|\wh{g_M}\|_4 \geq \|\wh{g}\|_4 - \exp(-k^{\eps/20})$ and $\|g_M\|_q \leq \|g\|_q$.
\end{lemma}

\begin{proof}
The inequality $\|g_M\|_q \leq \|g\|_q$ follows trivially from the definition of $g_M$. Concerning the $L^4$-norm of their Fourier transforms, we have by the triangle inequality, Hausdorff-Young inequality, and H\"{o}lder's inequality that
$$ \|\wh{g}\|_4 - \|\wh{g_M}\|_4 \leq \|\wh{g-g_M}\|_4 \leq \|g-g_M\|_{4/3} \leq \|g-g_M\|_{\infty}^{1/4} \|g-g_M\|_1^{3/4}. $$
Since
$$ \|g-g_M\|_{\infty} \leq g(M) = \exp(-M^2/A) \leq \exp(-k^{\eps/15}) $$
and
$$ \|g-g_M\|_1 \leq \|g\|_1 = \int_{-\infty}^{\infty} e^{-x^2/A} \dd x = (\pi A)^{1/2} \ll k, $$
it follows that
$$ \|g-g_M\|_{\infty}^{1/4} \|g-g_M\|_1^{3/4} \leq \exp(-k^{\eps/20}), $$
once $k$ is large enough in terms of $\eps$.
\end{proof}

Now we discretize $g_M$.
Define $f: \Z \rightarrow \R$ by $f(m) = g_M(m)$ for $m \in \Z$. Then $f$ is supported on $\{-M,\cdots,M\} \subset \{-k,\cdots,k\}$.

\begin{lemma}\label{lem:fg}
For $m \in \Z$, let $I_m = [m, m+1)$. Then
$$ \sup_{x \in I_m} |g_M(x) - f(m)| \ll k^{-1/2} f(m) $$
for every $m \in \Z$.
\end{lemma}

\begin{proof}
If $m \geq M$ or $m \leq -M-1$, then $f(m) = 0$ and $g_M(x) = 0$ for every $x \in I_m$, and hence the conclusion holds trivially. Now assume that $m \in \{-M, \cdots, M-1\}$, so that $I_m \subset [-M, M)$ and thus $g_M(x) = g(x)$ for $x \in I_m$. Hence for $x \in I_m$ we have
$$ |g_M(x) - f(m)| = |g(x) - g(m)| \leq \sup_{y \in [x,m]} |g'(y)| = \frac{2}{A} \sup_{y \in I_m} |yg(y)| \leq \frac{2}{A}(1+|m|) \sup_{y \in I_m} g(y) $$
Since
$$ g(m+1) = g(m) e^{-(2m+1)/A} \leq g(m) e^{2M/A} \leq 2g(m), $$
it follows that
$$ |g_M(x) - f(m)| \ll \frac{M}{A} g(m) \ll k^{-1/2} g(m). $$
\end{proof}

\begin{lemma}\label{lem:f-approx-gM}
We have $\|\wh{g_M}\|_4 \leq (1 + O(k^{-1/2})) \|\wh{f}\|_4$ and $\|g_M\|_q  = (1 + O(k^{-1/2})) \|f\|_q$.
\end{lemma}

\begin{proof}
Note that
\begin{equation*}
\begin{split}
\|\wh{g_M}\|_4^4 &= \int\int\int g_M(x_1)g_M(x_2)g_M(x_3)g_M(x_1+x_2-x_3) \dd x_1 \dd x_2 \dd x_3 \\
&= \sum_{a_1,a_2,a_3,a_4 \in \Z} \int\int\int g_M\vert_{I_{a_1}}(x_1) g_M\vert_{I_{a_2}}(x_2) g_M\vert_{I_{a_3}}(x_3) g_M\vert_{I_{a_4}}(x_1+x_2-x_3) \dd x_1 \dd x_2 \dd x_3.
\end{split}
\end{equation*}
By Lemma \ref{lem:fg} we have
$$ g_M\vert_{I_a}(x) = (1+O(k^{-1/2})) f(a) 1_{I_a}(x) $$
for any $a \in \Z$ and $x \in \R$. Hence
$$
\|\wh{g_M}\|_4^4 = \left(1 + O(k^{-1/2})\right) \sum_{a_1,a_2,a_3,a_4 \in \Z} f(a_1) f(a_2) f(a_3) f(a_4) I(a_1,a_2,a_3,a_4), 
$$
where
$$
I(a_1,a_2,a_3,a_4) = \int\int\int 1_{I_{a_1}}(x_1) 1_{I_{a_2}}(x_2) 1_{I_{a_3}}(x_3) 1_{I_{a_4}}(x_1+x_2-x_3) \dd x_1 \dd x_2 \dd x_3.
$$
By shifting the variables $x_1,x_2,x_3$ in the integral above, we see that
$$ I(a_1,a_2,a_3,a_4) = I(0, 0, 0, a_3+a_4-a_1-a_2). $$
It follows that 
$$ \|\wh{g_M}\|_4^4 = \left(1 + O(k^{-1/2})\right) \sum_{a \in \Z} I(0,0,0,a) \sum_{\substack{a_1,a_2,a_3,a_4 \in \Z \\ a_3+a_4-a_1-a_2=a}} f(a_1) f(a_2) f(a_3) f(a_4) $$
By Fourier analysis, we have
$$
\sum_{\substack{a_1,a_2,a_3,a_4 \in \Z \\ a_3+a_4-a_1-a_2=a}} f(a_1) f(a_2) f(a_3) f(a_4) = \int_0^1 |\wh{f}(\theta)|^4 e(a\theta) \dd \theta \leq \|f\|_4^4.
$$
Hence
$$ \|\wh{g_M}\|_4^4 \leq  \left(1 + O(k^{-1/2})\right) \|f\|_4^4 \sum_{a \in \Z} I(0,0,0,a) = \left(1 + O(k^{-1/2})\right) \|f\|_4^4. $$ 
This proves the first bound in the lemma. For the second bound concerning the $L^q$-norms, note that
$$  \|f\|_q^q - \|g_M\|_q^q = \sum_{a \in \Z} f(a)^q - \int_{-\infty}^{\infty} g_M(x)^q \dd x = \sum_{a \in \Z} \left(f(a)^q - \int_a^{a+1} g_M(x)^q \dd x\right). $$
By Lemma \ref{lem:fg} we have
$$ \int_a^{a+1} g_M(x)^q \dd x = (1+O(k^{-1/2})) f(a)^q $$
for every $a \in \Z$. It follows that
$$ \|f\|_q^q - \|g_M\|_q^q  = O\left(k^{-1/2} \sum_{a \in \Z} f(a)^q\right) = O\left(k^{-1/2} \|f\|_q^q\right). $$
This proves the second bound in the lemma.
\end{proof}

We may now complete the proof of Proposition \ref{prop:lower-bound} by combining the lemmas above. Indeed, by Lemmas \ref{lem:gM-approx-g} and \ref{lem:f-approx-gM} we have
$$ \|f\|_q \leq (1 + O(k^{-1/2})) \|g_M\|_q \leq (1 + O(k^{-1/2})) \|g\|_q $$
and 
$$ \|\wh{f}\|_4 \geq (1 - O(k^{-1/2})) \|\wh{g_M}\|_4 \geq (1 - O(k^{-1/2})) \left(\|\wh{g}\|_4 - \exp(-k^{\eps/20})\right). $$ 
Since $\|\wh{g}\|_4 \asymp A^{3/8}$, we have
$$ \|\wh{f}\|_4 \geq (1 - O(k^{-1/2})) \|\wh{g}\|_4. $$
It follows from Lemma \ref{lem:g-property} that
$$ \frac{\|\wh{f}\|_4}{\|f\|_q} \geq (1-O(k^{-1/2})) \frac{\|\wh{g}\|_4}{\|g\|_q} \geq (1 - O(k^{-1/2})) (1 + c\eps) > 1, $$ 
once $k$ is large enough in terms of $\eps$.

\section{Upper bound for $t_n$}\label{sec:upper-bound}

In this section we prove Proposition \ref{prop:upper-bound}. As explained in Section \ref{sec:outline}, a key ingredient is an approximate inverse theorem for Young's convolution inequality, Proposition \ref{prop:HY} below, which is a special case of results in \cite{CC, EisnerTao}. For completeness, we give a self-contained proof of it. In preparation for the proof, we start with establishing an approximate inverse theorem for H\"{o}lder's inequality, Lemma \ref{lem:Holder}, which is a special case of  \cite[Lemma 5.1]{EisnerTao}.

\subsection{Near equality in H\"{o}lder's inequality}

In this section, all implied constants are allowed to depend on the exponents $p,q,r$. 

\begin{lemma}\label{lem:pq}
Let $p,q \in (1,+\infty)$ be exponents with $1/p + 1/q = 1$. Let $a, b$ be non-negative reals. Suppose that
$$ \frac{a^p}{p} + \frac{b^q}{q} \leq (1 + \delta)ab $$
for some sufficiently small constant $\delta > 0$. Then $a^p = (1 + O(\delta^{1/2}))b^q$.
\end{lemma}

\begin{proof}
If $ab=0$, then the conclusion holds trivially. Henceforth assume that $a,b > 0$. By Taylor's theorem applied to the function $\psi(x) = \log x$ at the point $x_0 = a^p/p + b^q/q$, we have
$$ \psi(a^p) = \psi(x_0) + (a^p - x_0) \psi'(x_0) + \frac{1}{2} (a^p-x_0)^2 \psi''(\xi_1) $$
and
$$  \psi(b^q) = \psi(x_0) + (b^q - x_0) \psi'(x_0) + \frac{1}{2} (b^q-x_0)^2 \psi''(\xi_2) $$
for some $\xi_1,\xi_2$ lying between $a^p$ and $b^q$. Since
$$ a^p - x_0 = \frac{a^p - b^q}{q}, \ \ b^q-x_0 = \frac{b^q-a^p}{p}, $$
it follows that
$$ \frac{1}{p}\psi(a^p) + \frac{1}{q}\psi(b^q) = \psi(x_0) + \frac{(a^p-b^q)^2}{2pq^2} \psi''(\xi_1) + \frac{(a^p-b^q)^2}{2p^2q}\psi''(\xi_2). $$
Since $\psi''(x) = -1/x^2$, we have
$$ \psi''(\xi_i) \leq -\min\left(\frac{1}{a^{2p}}, \frac{1}{b^{2q}}\right). $$
From hypothesis we have
$$ \frac{1}{p}\psi(a^p) + \frac{1}{q}\psi(b^q) - \psi(x_0) = \log a + \log b - \log\left(\frac{a^p}{p} + \frac{b^q}{q}\right) \geq -\log(1+\delta) \geq - \delta. $$
Hence it follows that
$$ -\delta \leq -(a^p-b^q)^2\left(\frac{1}{2pq^2} + \frac{1}{2p^2q}\right) \min\left(\frac{1}{a^{2p}}, \frac{1}{b^{2q}}\right) = -\frac{(a^p-b^q)^2}{2pq} \min\left(\frac{1}{a^{2p}}, \frac{1}{b^{2q}}\right), $$
and thus
$$ (a^p - b^q)^2 \ll  \delta \max(a^{2p}, b^{2q}). $$
The desired conclusion follows immediately.
\end{proof}

\begin{lemma}\label{lem:pqr}
Let $p,q,r \in (1,+\infty)$ be exponents with $1/p + 1/q + 1/r = 1$. Let $a, b, c$ be non-negative reals. Suppose that
$$ \frac{a^p}{p} + \frac{b^q}{q} + \frac{c^r}{r} \leq (1 + \delta)abc $$
for some sufficiently small constant $\delta > 0$. Then $a^p = (1 + O(\delta^{1/2}))b^q = (1 + O(\delta^{1/2})c^r$.
\end{lemma}

\begin{proof}
We may assume that $abc > 0$, since otherwise the conclusion holds trivially.
Choose exponent $p' \in (1,+\infty)$ such that $1/p + 1/p' = 1$. Let
$$ d = \left(\frac{p'}{q} b^q + \frac{p'}{r} c^r\right)^{1/p'}. $$
Then
$$ \frac{a^p}{p} + \frac{b^q}{q} + \frac{c^r}{r} = \frac{a^p}{p} + \frac{d^{p'}}{p'} \geq ad. $$
From hypothesis it follows that $d \leq (1+\delta)bc$, which can be rewritten as
$$ \frac{x^{q'}}{q'} + \frac{y^{r'}}{r'} \leq (1+\delta)^{p'} xy, $$
where
$$ q' = \frac{q}{p'}, \ \ r' = \frac{r}{p'}, \ \ x = b^{p'}, \ \ y = c^{q'}. $$
Note that $1/q' + 1/r' = 1$. Hence by Lemma \ref{lem:pq} it follows that
$$ x^{q'} = (1 + O(\delta^{1/2})) y^{r'}, $$
which implies that
$$ b^q = (1 + O(\delta^{1/2})) c^r. $$
Similarly, one can also prove that $a^p = (1 + O(\delta^{1/2})) c^r$.
\end{proof}

\begin{lemma}\label{lem:Holder}
Let $p,q,r \in (1,+\infty)$ be exponents with $1/p + 1/q + 1/r = 1$.
Let $a_1,\cdots,a_n$, $b_1,\cdots,b_n$, $c_1,\cdots,c_n$ be non-negative reals such that
$$ \sum_{i=1}^n a_i^p = \sum_{i=1}^n b_i^q = \sum_{i=1}^n c_i^r = 1. $$
Suppose that
$$ \sum_{i=1}^n a_ib_ic_i \geq 1 - \delta $$
for some sufficiently small constant $\delta > 0$. Then we have
$$ a_i^p = (1 + O(\delta^{1/4})) b_i^q = (1 + O(\delta^{1/4})) c_i^r $$
for each $i$ outside an exceptional set $E$ satisfying
$$ \sum_{i \in E} (a_i^p + b_i^q + c_i^r) \ll \delta^{1/2}. $$
\end{lemma}

\begin{proof}
For each $i$ we have
$$ a_ib_ic_i \leq \frac{a_i^p}{p} + \frac{b_i^q}{q} + \frac{c_i^r}{r}. $$
Let $E \subset \{1,2,\cdots,n\}$ be the exceptional set of indices $i$ such that
$$  \frac{a_i^p}{p} + \frac{b_i^q}{q} + \frac{c_i^r}{r} \geq (1+\delta^{1/2}) a_ib_ic_i. $$
Then
$$ \delta \geq \sum_{i=1}^n \left(\frac{a_i^p}{p} + \frac{b_i^q}{q} + \frac{c_i^r}{r} - a_ib_ic_i\right) \gg \delta^{1/2} \sum_{i \in E} \left(\frac{a_i^p}{p} + \frac{b_i^q}{q} + \frac{c_i^r}{r}\right), $$
and hence $\sum_{i \in E} (a_i^p + b_i^q + c_i^r) \ll \delta^{1/2}$. For $i \notin E$, Lemma \ref{lem:pqr} implies that 
$$ a_i^p = (1 + O(\delta^{1/4})) b_i^q = (1 + O(\delta^{1/4})) c_i^r. $$
This concludes the proof.
\end{proof}

\subsection{Near equality in Young's inequality}

In this section, all implied constants are allowed to depend on the exponents $p,q,r$. Before proving the approximate inverse of Young's inequality, we need the following standard result in additive combinatorics.

\begin{lemma}\label{lem:doubling1}
Let $G$ be an abelian group and let $X, Y \subset G$ be finite subsets with $|X| = |Y| = N$. Let $\eps \in (0, 1/20)$ and let $\delta > 0$ be sufficiently small in terms of $\eps$. Let $M \subset X \times Y$ be a subset with $|M| \geq (1-\delta) N^2$. Suppose that the restricted sumset
$$ X+_MY := \{x+y: (x,y) \in M\} $$
has size at most $(1+\eps)N$. Then there exists a coset $x+H$ of a subgroup $H \subset G$ such that $|X \setminus (x+H)| \leq \eps N$ and $|(x+H) \setminus X| \leq 3\eps N$.
\end{lemma}

\begin{proof}
By an almost-all version of the Balog-Szemeredi-Gowers theorem as in \cite[Theorem 1.1]{Shao} (see also \cite[Theorem 1.1]{ShaoXu} for a version with $G = \Z$ and \cite[Theorem 3.3]{CCSW} for an asymmetric version), one can find subsets $X' \subset X$ and $Y' \subset Y$ such that
$$ |X'| \geq (1-\eps)N, \ \ |Y'| \geq (1-\eps)N, \ \ |X'+Y'| \leq |X+_MY| + \eps N \leq (1+2\eps)N. $$
By Kneser's theorem \cite{Kneser} (see \cite[Theorem 5.5]{TaoVu}), we have
$$ |X'+Y'| \geq |X'| + |Y'| - |H|, $$
where $H \subset G$ is the subgroup defined by
$$ H = \{h \in G: X'+Y'+h = X'+Y'\}. $$
It follows that $|H| \geq (1-4\eps)N$ and hence $|X'+Y'| < 2|H|$.
Since $X'+Y'$ is the union of cosets of $H$, it must be a single coset of $H$, and thus $X'$ is contained in a single coset $x+H$ of $H$. Hence
$$ |X\setminus (x+H)| \leq |X\setminus X'| \leq \eps N $$
and
$$ |(x+H) \setminus X| \leq |(x+H) \setminus X'| = |X'+Y'| - |X'| \leq 3\eps N. $$
\end{proof}

\begin{proposition}\label{prop:HY}
Let $p,q,r \in (1,+\infty)$ be exponents with $1/p + 1/q = 1 + 1/r$.
Let $f,g: \Z\rightarrow\C$ be finitely-supported functions such that $\|f\|_p = \|g\|_q = 1$. Suppose that
$$ \|f*g\|_r \geq 1-\delta $$
for some sufficiently small constant $\delta > 0$. Then there exists a singleton set $\{x_0\}$ for some $x_0 \in \Z$ such that 
$$ \|f - f(x_0)1_{\{x_0\}}\|_p^p \ll \delta^{1/8}. $$
\end{proposition}

\begin{proof}
By replacing $f,g$ by $|f|, |g|$, we may assume that $f,g$ takes non-negative real values.
For every $x \in \Z$ we have
$$ (f*g)(x) = \sum_{y\in\Z} f(x-y) g(y) = \sum_{y \in \Z} f(x-y)^{p/r} g(y)^{q/r} \cdot f(x-y)^{(r-p)/r} \cdot g(y)^{(r-q)/r}.  $$
By H\"{o}lder's inequality, we have
\begin{equation}\label{HY-eq1} 
(f*g)(x) \leq \left(\sum_{y \in \Z} f(x-y)^p g(y)^q \right)^{\frac{1}{r}} \left(\sum_{y \in \Z} f(x-y)^p\right)^{\frac{r-p}{pr}} \left(\sum_{y \in \Z} g(y)^q\right)^{\frac{r-q}{qr}}.
\end{equation}
Since $\|f\|_p = \|g\|_q = 1$, it follows that
$$ (f*g)(x)^r \leq \sum_{y \in \Z} f(x-y)^p g(y)^q. $$
Let $E_1 \subset \Z$ be the exceptional set of $x \in \Z$ such that
$$ (f*g)(x)^r \leq (1 - \delta^{1/2}) \sum_{y \in \Z} f(x-y)^p g(y)^q. $$
From hypothesis we have
$$ 1 - (1-\delta)^r \geq \sum_{x \in \Z} \left(\sum_{y \in \Z} f(x-y)^p g(y)^q - (f*g)(x)^r\right) \geq \delta^{1/2} \sum_{x \in E_1} \sum_{y \in \Z} f(x-y)^p g(y)^q $$
and hence
\begin{equation}\label{HY-exceptional1} 
\sum_{(x,y) \in E_1\times \Z} f(x-y)^p g(y)^q \ll \delta^{1/2}.
\end{equation}
For each $x \notin E_1$, we have almost equality in \eqref{HY-eq1}, and hence by Lemma \ref{lem:Holder} applied to the three sequences
$$ a_x(y) = \frac{f(x-y)^{p/r} g(y)^{q/r}}{h(x)^{1/r}}, \ \ b_x(y) = f(x-y)^{(r-p)/r}, \ \ c_x(y) = g(y)^{(r-q)/r}, $$
where
$$ h(x) = \sum_{z \in \Z} f(x-z)^p g(z)^q, $$
we conclude that
\begin{equation}\label{HY-eq2} 
\frac{f(x-y)^pg(y)^q}{h(x)} = (1 + O(\delta^{1/8})) f(x-y)^p = (1 + O(\delta^{1/8})) g(y)^q
\end{equation}
for each $y$ outside an exceptional set $E_2(x)$ satisfying
\begin{equation}\label{HY-exceptional2} 
\sum_{y \in E_2(x)} f(x-y)^p g(y)^q \ll \delta^{1/4}h(x). 
\end{equation}
Define 
$$ E = (E_1 \times \Z) \cup \{(x, y) \in \Z \times \Z: x \notin E_1, y \in E_2(x)\}. $$
Then from \eqref{HY-exceptional1} and \eqref{HY-exceptional2} it follows that
$$ \sum_{(x,y) \in E} f(x-y)^p g(y)^q \ll \delta^{1/2} + \delta^{1/4} \sum_{x \in \Z}h(x) \ll \delta^{1/4}, $$
and \eqref{HY-eq2} holds for every $(x,y) \notin E$.

Now make a change of variables and consider 
$$ E' = \{(x,y) \in \Z \times \Z: (x+y, y) \in E\}. $$
Then
\begin{equation}\label{HY-exceptional'} 
\sum_{(x,y) \in E'} f(x)^p g(y)^q = \sum_{(x,y) \in E} f(x-y)^pg(y)^q \ll \delta^{1/4}, 
\end{equation}
and we have 
\begin{equation}\label{HY-eq3} 
\frac{f(x)^p g(y)^q}{h(x+y)} = (1 + O(\delta^{1/8}))f(x)^p = (1 + O(\delta^{1/8})) g(y)^q 
\end{equation}
for every $(x,y) \notin E'$. Let $X \subset \Z$ be the set of $x \in \Z$ such that 
$$ \sum_{y: (x,y) \in E'} g(y)^q \leq \delta^{1/8}. $$ 
Then from \eqref{HY-exceptional'} it follows that
$$ \delta^{1/4} \gg \sum_{x \notin X} f(x)^p \sum_{y: (x,y) \in E'} g(y)^q \geq \delta^{1/8} \sum_{x \notin X} f(x)^p, $$ 
and hence
\begin{equation}\label{HY-exceptionalX} 
\sum_{x\notin X} f(x)^p \ll \delta^{1/8}.
\end{equation}
For every $x_1, x_2 \in X$, since
$$ \sum_{y: (x_1,y) \in E'} g(y)^q + \sum_{y: (x_2,y) \in E'} g(y)^q \ll \delta^{1/8}, $$
there exists $y \in \Z$ such that $(x_1,y) \notin E'$ and $(x_2,y)\notin E'$. By \eqref{HY-eq3} we have
$$ f(x_1)^p = (1 + O(\delta^{1/8})) g(y)^q, \ \ f(x_2)^p = (1 + O(\delta^{1/8})) g(y)^q. $$
We conclude that there exists a constant $a \in \R$ such that
$$ 
f(x) = (1 + O(\delta^{1/8})) a 
$$
for every $x \in X$. Moreover, since 
$$ 1 = \sum_{x \in \Z} f(x)^p = \sum_{x \in X} f(x)^p + O(\delta^{1/8}) = (1 + O(\delta^{1/8})) a^p |X| + O(\delta^{1/8}), $$
we have
$$ |X| = (1 + O(\delta^{1/8})a^{-p}. $$
By symmetry, we may also conclude the existence of a constant $b \in \R$ such that
$$ g(y) = (1 + O(\delta^{1/8})) b $$
for every $y \in Y$, where $Y \subset \Z$ is a subset satisfying
$$ |Y| = (1 + O(\delta^{1/8})b^{-q}.  $$
We now return to using the first part of \eqref{HY-eq3} for $(x,y) \in M := (X \times Y)\setminus E'$. First note from \eqref{HY-exceptional'} that
$$ \delta^{1/4} \gg \sum_{(x,y) \in (X\times Y) \cap E'} f(x)^p g(y)^q \gg a^p b^q \cdot |(X\times Y) \cap E'| \gg |X|^{-1}|Y|^{-1} \cdot |(X\times Y) \cap E'|, $$
and hence
$$ |M| \geq (1 - O(\delta^{1/4})) |X||Y|. $$
For $(x,y) \in M$, \eqref{HY-eq3} implies that
$$ \frac{a^pb^q}{h(x+y)} = (1 + O(\delta^{1/8}) a^p = (1 + O(\delta^{1/8})) b^q. $$
In particular, since $M$ is non-empty, we have $a^p = (1 + O(\delta^{1/8})) b^q$ and hence $|X| = (1 + O(\delta^{1/8})) |Y|$. Moreover, for $s \in X+_MY$ we have
$$ h(s) = (1  + O(\delta^{1/8})) a^p. $$
Since $\sum_{s \in \Z}h(s) = 1$, we have
$$ 1 \geq \sum_{s \in X+_MY} h(s) = (1 - O(\delta^{1/8}))a^p \cdot |X+_MY|, $$
and hence
$$ |X+_MY| \leq (1 + O(\delta^{1/8})) |X|. $$
We now apply Lemma \ref{lem:doubling1} with $\eps = 1/100$ (say), after possibly shrinking one of $X, Y$ slightly so that $|X| = |Y|$, to conclude that there exists a coset $x_0 + H$ of a subgroup $H \subset \Z$ such that
$$ |X \setminus (x_0+H)| \leq \frac{1}{10}|X|, \ \ |(x_0+H) \setminus X| \leq \frac{1}{10}|X|. $$
The only finite subgroup of $\Z$ is $H = \{0\}$, and hence it must be that $X = \{x_0\}$. The desired conclusions follow immediately from \eqref{HY-exceptionalX}.
\end{proof}

\subsection{Proof of Proposition \ref{prop:upper-bound}}

Let $f: \Z \rightarrow \R$ be a function which is supported on a set of size $n \geq 2$. By replacing $f$ by $|f|$, we may assume that $f$ takes non-negative real values. Let $\delta > 0$ be a sufficiently small absolute constant and let
$$ q = \frac{4}{3 - \log_n(1+\delta)} $$
First consider the case when
$$ \|\widehat{f}\|_4 \leq (1 - \delta) \|f\|_{4/3}. $$
By H\"{o}lder's inequality (see \eqref{eq:q-4/3}), we have
$$ \|f\|_{4/3} \leq n^{3/4-1/q} \|f\|_q = (1+\delta)^{1/4} \|f\|_q. $$
It follows that
$$ \|\widehat{f}\|_4 \leq (1-\delta) (1+\delta)^{1/4} \|f\|_q \leq \|f\|_q. $$

Now suppose that
$$ \|\widehat{f}\|_4 \geq (1 - \delta) \|f\|_{4/3}. $$
By normalization we may assume that $\|f\|_{4/3} = 1$, and thus $\|f*f\|_2 = \|\wh{f}\|_4^2 \geq 1-2\delta$.
By Proposition \ref{prop:HY}, there exists $x_0 \in \Z$ such that
\begin{equation}\label{proof-upper-bound1} 
\|f - f(x_0)1_{\{x_0\}}\|_{4/3} \ll \delta^{1/20}.
\end{equation}
By translation we may assume that $x_0=0$, and we may write $f$ in the form $f = f(0)\delta_0 + g$, where $\delta_0$ is the Kronecker delta function and $g(0) = 0$. Let $x = f(0)$ and $y = \|g\|_{4/3}$. Since $\|f\|_{4/3}=1$ we have
$$ x^{4/3} + y^{4/3} = 1. $$
From \eqref{proof-upper-bound1} we have
$$ y = O(\delta^{1/20}), \ \ x = 1 - O(\delta^{1/20}). $$
In particular we have $y/x \leq 0.01$.
Since $f*f = x^2\delta_0 + 2xg + g*g$, we have
$$ \|f*f\|_2 \leq x\|x\delta_0 + 2g\|_2 + \|g*g\|_2 = x \sqrt{\|x\delta_0\|_2^2 + \|2g\|_2^2} + \|g*g\|_2. $$
Using the inequalities $\|g\|_2 \leq \|g\|_{4/3} = y$ and $\|g*g\|_2 \leq \|g\|_{4/3}^2 = y^2$, we obtain
$$ \|f*f\|_2 \leq x\sqrt{x^2 + 4y^2} + y^2 = x^2 \sqrt{1 + \frac{4y^2}{x^2}} + y^2.  $$
Since $\sqrt{1+\lambda} \leq 1 + \lambda/2$ for $\lambda \geq 0$, it follows that
$$ \|f*f\|_2 \leq x^2\left(1 + \frac{2y^2}{x^2}\right) + y^2 = x^2 + 3y^2. $$
On the other hand, note that
$$ \|f\|_q = (x^q + \|g\|_q^q)^{1/q}. $$
Since $g$ is supported on a set of size $n$, by H\"{o}lder's inequality we have
$$ \|g\|_{4/3} \leq n^{3/4 - 1/q} \|g\|_q = (1+\delta)^{1/4} \|g\|_q. $$
By choosing $\delta>0$ to be small enough, we have $\|g\|_q^q \geq 0.9 \|g\|_{4/3}^q = 0.9y^q$, and hence
$$ \|f\|_q^2 \geq (x^q + 0.9y^q)^{2/q} = x^2\left(1 + \frac{0.9y^q}{x^q}\right)^{2/q}. $$
Since $4/3 \leq q \leq 3/2$, we have $(1+\lambda)^{2/q} \geq 1+\lambda \geq 1+4\lambda^{2/q}$ for $0 \leq \lambda \leq 1/64$. Hence
$$ \|f\|_q^2 \geq x^2\left(1 + 4 \cdot 0.9^{2/q} \cdot \frac{y^2}{x^2}\right) \geq x^2 + 3y^2. $$
It follows that $\|f*f\|_2 \leq \|f\|_q^2$, as desired.

\section{Proof of Theorem \ref{main-thm2}}\label{sec:thm2}

Let $n \geq 3$ be a positive integer and let $I$ be the interval
$$ I = \left\{-\Bigl\lfloor \frac{n-1}{2}\Bigr\rfloor, \cdots, \Bigl\lfloor \frac{n}{2}\Bigr\rfloor\right\} $$
which has length $n$.
In view of Proposition \ref{prop:tool}, it suffices to construct a function $f: I \rightarrow \R$ such that $|\wh{f}\|_4 > \|f\|_q$,
where
$$ q = \frac{4}{\log_n \frac{2n^3+n}{3}}. $$
We take $f = 1_I + \eps\delta_0$ for some small $\eps > 0$, where $\delta_0$ is the Kronecker delta function. Note that $\|\wh{1_I}\|_4 = \|1_I\|_q$, and we will show that the small adjustment from $1_I$ to $f$ swings the inequality in the desired direction.

First note that
$$ \|f\|_q^q = n-1 + (1+\eps)^q = n + q\eps + O(\eps^2). $$
Hence
$$ \|f\|_q^4 = n^{4/q} \left(1 + \frac{q\eps}{n} + O\left(\frac{\eps^2}{n}\right)\right)^{4/q} = n^{4/q} \left(1 + \frac{4\eps}{n} + O\left(\frac{\eps^2}{n}\right)\right). $$
Since $n^{4/q} = (2n^3+n)/3$, it follows that
\begin{equation}\label{thm2-eq1} 
\|f\|_q^4 = \frac{1}{3}(2n^3+n) + \frac{4}{3}(2n^2+1)\eps + O(n^2\eps^2). 
\end{equation}
Now consider the convolution
$$ f*f = 1_I*1_I + 2\eps 1_I + \eps^2\delta_0. $$
We have
\begin{equation*}
\begin{split} 
\|f*f\|_2^2 &= \sum_{a \notin I} 1_I*1_I(a)^2 + \sum_{a \in I\setminus \{0\}} (1_I*1_I(a) + 2\eps)^2 + (1_I*1_I(0) + 2\eps + \eps^2)^2 \\
&= \sum_{a \notin I} 1_I*1_I(a)^2 + \sum_{a \in I} (1_I*1_I(a) + 2\eps)^2 + O(n\eps^2) \\
&= \sum_{a \in \Z} 1_I*1_I(a)^2 + 4\eps \sum_{a \in I} 1_I*1_I(a) + O(n\eps^2)
\end{split}
\end{equation*}
One can compute that
$$ \sum_{a \in \Z}1_I*1_I(a)^2 = E(I) = \frac{1}{3}(2n^3+n) $$
and 
$$ \sum_{a \in I}1_I*1_I(a) = \Bigl\lceil \frac{3n^2}{4}\Bigr\rceil \geq \frac{3n^2}{4} $$
Hence
\begin{equation}\label{thm2-eq2} 
\|f*f\|_2^2 \geq \frac{1}{3}(2n^3+n) + 3n^2\eps + O(n\eps^2). 
\end{equation}
Comparing \eqref{thm2-eq1} with \eqref{thm2-eq2} and noting that
$$ 3n^2 > \frac{4}{3}(2n^2+1) $$
for every $n \geq 3$, we conclude that
$$ \|f*f\|_2^2 > \|f\|_q^4 $$
for sufficiently small $\eps > 0$. This completes the proof.

\bibliographystyle{plain}
\bibliography{biblio}

\end{document}